\documentclass[12pt]{amsart}
\usepackage{amsmath, amssymb, epsfig}
\usepackage{bm}
\usepackage{mathrsfs}
\usepackage{color}

\newlength{\algorithmwidth}
\theoremstyle{plain}
\newtheorem{theorem}{Theorem}[section]

\newtheorem{lemma}[theorem]{Lemma}

\theoremstyle{definition}

\theoremstyle{remark}

\numberwithin{equation}{section}

\newcommand{\<}{\left\langle}
\renewcommand{\>}{\right\rangle}
\newcommand{\bigO}{\mathrm{O}}

\DeclareMathOperator*{\supp}{supp}

\DeclareMathOperator*{\trace}{trace}

\DeclareMathOperator*{\Null}{Null}

\DeclareMathOperator*{\argmin}{arg min}

\newcommand{\defby}{\overset{\mathrm{\scriptscriptstyle{def}}}{=}}

\def \P {\mathbb{P}}

\def \A {\mathcal{A}}
\def \M {\mathcal{M}}
\def \G {\mathcal{G}}
\def \S {\mathcal{S}}
\def \H {\mathcal{H}}
\def \R {\mathcal{R}}
\def \V {\mathcal{V}}

\def \W {\mathcal{W}}
\def \eps {\varepsilon}

\def \rank {{\rm rank }}

\begin{document}
\bibliographystyle{plain}
\setlength{\parindent}{0in}
\parskip 7.2pt

\title[]{Unicity conditions for low-rank matrix recovery}
\author{Y. C. Eldar, D. Needell, and Y. Plan}

\date{\today}


\begin{abstract}
Low-rank matrix recovery addresses the problem of recovering an unknown low-rank matrix from few linear measurements.  Nuclear-norm minimization is a tractible approach with a recent surge of strong theoretical backing.  Analagous to the theory of compressed sensing, these results have required random measurements.  For example, $m \geq Cnr$ Gaussian measurements are sufficient to recover any rank-$r$ $n\times n$ matrix with high probability.  In this paper we address the theoretical question of how many measurements are needed via any method whatsoever --- tractible or not.  We show that for a family of random measurement ensembles, $m \geq 4nr - 4r^2$ measurements are sufficient to guarantee that no rank-$2r$ matrix lies in the null space of the measurement operator with probability one.  This is a necessary and sufficient condition to ensure uniform recovery of all rank-$r$ matrices by rank minimization.  Furthermore, this value of $m$ precisely matches the dimension of the manifold of all rank-$2r$ matrices.  We also prove that for a fixed rank-$r$ matrix, $m \geq 2nr - r^2 + 1$ random measurements are enough to guarantee recovery using rank minimization.  These results give a benchmark to which we may compare the efficacy of nuclear-norm minimization.  
\end{abstract}

\maketitle

\section{Introduction}

In the compressed sensing problem, one wishes to recover an unknown vector $x\in\mathbb{R}^n$ from few linear measurements of the form $y=Ax\in\mathbb{R}^m$ where $A$ is an $m\times n$ measurement matrix and $m \ll n$ (see e.g.~\cite{CW08:Intro,B07:Compressive,CSwebpage} for tutorials on compressed sensing).  This problem is clearly ill-posed until additional assumptions are enforced.  A common assumption is that $x$ is \textit{$s$-sparse}: the support of $x$ is small, $\|x\|_0 = |\supp(x)| \leq s \ll n$.  If $A$ is injective on all $s$-sparse vectors then when $x$ is $s$-sparse, $x$ will be the solution to

\begin{equation}\tag{$L_0$}
\hat{x} = \argmin_w \|w\|_0 \quad\text{such that}\quad Aw = y.
\end{equation}

Moreover, for a matrix $A$ to be injective on all $s$-sparse vectors, we precisely require that its null space be disjoint from the set of all $2s$-sparse vectors.  Since there are many classes of matrices satisfying this property with $m=2s$ rows (see e.g.~\cite[Theorem 1.1]{CRT06:Robust-Uncertainty}), this shows that only $2s$ measurements are required to recover all $s$-sparse vectors $x\in\mathbb{R}^d$!  If we consider the problem of weak recovery, where we only wish to recover one \textit{fixed} vector $x$, $s+1$ measurements suffice.  These are of course \textit{theoretical} requirements, as the problem $(L_0)$ is a combinatorial optimization problem and is NP-Hard in general (see Sec. 9.2.2 of ~\cite{M05:Data}).

Work in the field of compressed sensing has however provided us with numerically feasible methods for sparse signal recovery.  One such method is $\ell_1$-minimization which is a relaxation of $(L_0)$:

\begin{equation}\tag{$L_1$}
\hat{x} = \argmin_z \|z\|_1 \quad\text{such that}\quad Az = y.
\end{equation}

It has been shown that for certain measurement matrices $A$, $(L_0)$ and $(L_1)$ are equivalent~\cite{DS89:Uncertainty-Principles,CT05:Decoding}.  These measurement ensembles can be taken randomly (for example, $A$ can be chosen to have Gaussian entries), and require $m \geq \bigO(s\log (n/s))$ measurements to guarantee reconstruction of all $s$-sparse vectors.  Thus we require slightly more measurements (from $2s$ to $Cs\log (n/s)$) but can recover via the problem $(L_1)$ which is numerically feasible by linear programming methods.  For weak recovery we need only slightly fewer measurements, see~\cite{DT06:Counting-Faces} for precise thresholds.

\subsection{Low-Rank Matrix Recovery}

A related problem to compressed sensing is the problem of low-rank matrix recovery, for which many results have been obtained (see e.g.~\cite{cai2008singular,candes2010matrix,candes2009exact,dai2010set,meka2009guaranteed,lee2010admira,ma2009fixed,recht2007guaranteed,G09:Recovering,candes2009tight}).  In this setting, we would like to recover a matrix $M$ from few of its linear measurements.  The measurement operator is of the form $\A: \mathbb{R}^{n\times n} \rightarrow \mathbb{R}^m$ and acts on a matrix $M$ by $(\A(M))_i = \langle A_i, M \rangle$ where $A_i$ are $n\times n$ matrices and $\langle \cdot , \cdot \rangle$ denotes the usual matrix inner product:

$$
\langle A, B \rangle \defby \trace(A^*B).
$$

Given the measurements $\A(M)\in \mathbb{R}^m$, we wish to recover the matrix $M\in\mathbb{R}^{n\times n}$.  This is of course again ill-posed for small $m$ in general.  However, if we operate under the assumption that $M$ has low rank then the problem can be made well-posed.  The question then becomes how large does $m$ need to be in order to guarantee recovery of rank-$r$ matrices and how does one recover such a matrix?  Analagous to the program $(L_0)$, one can consider solving

\begin{equation}\label{eqn:MCtheory}
\hat{X} = \argmin_X \rank(X) \quad\text{such that}\quad \A(X) = \A(M).
\end{equation}

This is simply a uniqueness problem; when is $M$ the unique low rank matrix having these measurements?  However, as in the case of $(L_0)$, the problem~\eqref{eqn:MCtheory} is intractible in general.

Instead of solving~\eqref{eqn:MCtheory}, we are often interested in a tractible method which provides worst-case guarantees; that is, guarantees which apply to \textit{all} rank-$r$ matrices whether arbitrary or adversarial.  A simple observation allows one to select an appropriate relaxation of~\eqref{eqn:MCtheory} that will do just this.  The rank of a matrix is the number of non-zero singular values.  That is, if $\sigma$ is the vector of singular values of $M$, then $\rank(M) = \|\sigma\|_0$.  Thus a natural relaxation would be to minimize $\|\sigma\|_1$.  We thus consider the minimization problem

\begin{equation}\label{eqn:MC}
\hat{X} = \argmin_X \|X\|_* \quad\text{such that}\quad \A(X) = \A(M),
\end{equation}

where $\|\cdot\|_*$ denotes the \textit{nuclear norm} which is defined by

$$
\|X\|_* = \trace(\sqrt{X^* X)}) = \sum_{i=1}^n \sigma_i.
$$

The program~\eqref{eqn:MC} can be cast as a semidefinite program (SDP) and is therefore numerically feasible.  {Moreover, it has been shown~\cite{negahban2010unified, recht2007guaranteed,oymak2010new, candes2009tight} that $m\geq Cnr$ measurements suffice to recover any $n\times n$ rank-$r$ matrix via~\eqref{eqn:MC}.}   

A question that does not appear to have been previously addressed is, how many measurements suffice to recover rank-$r$ matrices via the more natural (yet intractible) method~\eqref{eqn:MCtheory}?  In the compressed sensing setting, this question was easy to answer because the set of $s$-sparse vectors is the union of a finite number of linear subspaces.  In the matrix recovery problem, however, this question has remained unresolved.  Answering this question would not only fill a gap in the literature but also give theoretical bounds on the number of measurements required for low-rank matrix recovery against which those for problem~\eqref{eqn:MC} may be compared.  In the case of compressed sensing for example, it is clear that to use a tractible method we pay in the number of measurements by a factor of $\log (n/s)$.  What is this factor in the matrix recovery framework?  How good is nuclear-norm minimization?  These are the issues we address in this work.  

In this paper we prove that $4nr - 4r^2$ measurements are sufficient to recover all rank$-r$ $n\times n$ matrices using rank minimization almost surely.  To recover a fixed rank$-r$ $n\times n$ matrix with probability one, we show that only $2nr-r^2+1$ measurements are required.  We then compare our results to nuclear norm minimization and show that rank minimization requires less measurements, but only by a constant factor.

\section{Uniqueness Results}\label{thms}

In this section we provide a detailed summary of our main results.

We consider random operators $\A$, and first ask that for any rank-$r$ matrix $M$, the solution to~\eqref{eqn:MCtheory} is $\hat{X} = M$ with probability one.  If this were not the case, then there would be some matrices $M$ and $M'$ each with rank-$r$ or less such that $\A(M) = \A(M')$.  This means that the rank-$2r$ (or less) matrix $M - M'$ is in the null space of $\A$.  Therefore, to guarantee that~\eqref{eqn:MCtheory} reconstructs all rank-$r$ matrices, a necessary and sufficient condition is that there are no rank-$2r$ (or less) matrices in the null space of $\A$.  Thus we examine the following subset of $\mathbb{R}^{n\times n}$:
\begin{equation}\label{defS}
\R' = \{X\in\mathbb{R}^{n\times n} : \rank(X) = 2r\}.
\end{equation}
 We first wish to compute how large $m$ must be so that the null space of $\A$ is disjoint from $\R$.  We will then repeat this argument for smaller values of the rank.

It is well known that $\R'$ is a manifold with $4nr-4r^2$ dimensions.  Is $m \geq 4nr-4r^2$ sufficient to guarantee uniform recovery?  We will show that the answer is yes!  This is summarized by the following theorem.

Below, we call $\A$ a Gaussian operator if each $A_i$ is independent with i.i.d. Gaussian entries.

\begin{theorem}[Strong Recovery]\label{thm:main}
Let $r \leq n/2$. When $\A: \mathbb{R}^{n\times n} \rightarrow \mathbb{R}^m$ is a Gaussian operator with $m \geq 4nr-4r^2$,  problem~\eqref{eqn:MCtheory} recovers all rank-$r$ matrices with probability $1$. 
\end{theorem}

\begin{remarks}
{\bfseries \\1.} We actually prove a more general result in Theorem~\ref{thm:gen}.  In this result we consider any random linear operator $\A$ which takes $m\geq d+1$ measurements $\langle A_i, X\rangle$ where $\langle A_i, X\rangle$ are independent and do not concentrate around zero.  Then Theorem~\ref{thm:gen} shows that any $d$-dimensional continuously differentiable manifold over the set of real matrices is disjoint (except possibly at the origin) from the null space of $\A$.  Theorem~\ref{thm:main} will follow as a consequence.

{\bfseries 2.} {We consider real-valued matrices but our method can easily be extended to complex-valued matrices as well.}
\end{remarks}

Our proof technique also allows us to provide a bound on the number of measurements required for weak recovery.  Recall that in this framework we are interested in recovering \textit{one fixed matrix} $M$ with high probability.  Since $M$ is fixed, we require only that for all rank-$r$ matrices $X\neq M$ that $X-M$ is not in the null space of $\A$.  The set of all rank-$r$ matrices is a manifold of dimension $2nr-r^2$.  Recall that in compressed sensing for weak recovery we a require number of measurements equal to at least one more than the sparsity level.  The following result shows that for weak recovery of low-rank matrices we require a number of measurements at least one more than the dimension of the manifold of all rank-$r$ matrices.

\begin{theorem}[Weak Recovery]\label{thm:weak}
Fix a rank$-r$ $n\times n$ real matrix $M$.  When $\A: \mathbb{R}^{n\times n} \rightarrow \mathbb{R}^m$ is a Gaussian operator with $m \geq 2nr-r^2+1$, problem~\eqref{eqn:MCtheory} recovers the matrix $M$ with probability $1$. 
\end{theorem}

As we will see in Section~\ref{sec:discuss}, this theorem allows the comparison of rank minimization to the theoretical and empirical results of nuclear-norm minimization in the Gaussian setting.

We prove these results in the next section.  In Section~\ref{sec:discuss} we discuss the tightness of these bounds and compare them with results for nuclear-norm minimization.


\section{General Results and Proofs}

On our way to proving our main results, Theorems~\ref{thm:main} and~\ref{thm:weak}, we will prove a more general result about arbitrary manifolds of real matrices.  This result can be extended even further by considering manifolds over more arbitrary Banach spaces and following our proof.  For convenience we will restrict ourselves to the Banach space of real matrices.  Below, a continuously differentiable manifold is a manifold that may be equipped with a class of atlases having transition maps which are all $C^1$-diffeomorphisms.

\begin{theorem}\label{thm:gen}
Let $\R$ be a $d$-dimensional continuously differentiable manifold over the set of $n\times n$ real matrices. Suppose we take $m\geq d+1$ measurements of the form $\langle A_i, X\rangle$ for $X\in \R$, and define the operator $\A : \R \rightarrow \mathbb{R}^m$ which takes these measurements, $\A : X \mapsto y$ with $y_i = \langle A_i, X\rangle$.  Assume that there exists a constant $C=C(n)$ such that $\mathbb{P}(| \langle A_i, X \rangle | < \eps) < C\eps$ for every $X$ with $\|X\|_F=1$.  Further assume that for each $X\neq 0$ that the random variables $\{\langle A_i, X \rangle\}$ are independent.  Then with probability $1$,
$$
\Null(\A) \cap \R \backslash \{0\} = \emptyset.
$$ 
\end{theorem}

\begin{remarks}
{\bfseries \\1.} The requirement that $\mathbb{P}(| \langle A_i, X \rangle | < \eps) < C\eps$ says that the densities of $\langle A_i, X \rangle$ do not spike at the origin.  A sufficient condition for this to hold for every $X$ with $\|X\|_F=1$ is that each $A_i$ has i.i.d. entries with continuous density.

{\bfseries 2.} The requirement $m \geq d+1$ is tight in the sense that the result does not generally hold for $m \leq d$.  For example, take $\R$ to be the intersection of any $(d+1)$-dimensional linear subspace of $\mathbb{R}^{n\times n}$ with the unit sphere.  Then it is not hard to show that $\Null(A)\cap\R\backslash\{0\} \neq \emptyset$ for any linear operator $\A : \mathbb{R}^{n\times n} \rightarrow \mathbb{R}^m$ as long as $m\leq d$. 

\end{remarks}

To prove this result we will utilize a well-known fact about covering numbers.  For a set $B$, norm $\|\cdot\|$ and value $\eps$, we denote by $N(B, \|\cdot\|, \eps)$ the smallest number of balls (with respect to the norm $\|\cdot\|$) of radius $\eps$ whose union contains $B$.  This number is called a \textit{covering number}, and the set of balls covering the space (or more precisely the center of these balls) is called an $\eps$-\textit{net}.  A bound on the covering number for the unit ball under the Euclidean norm $\|\cdot\|_2$ is now well known (see e.g. Ch. 13 of~\cite{LGM96:Const}):

\begin{lemma}\label{lem:CN}
For any $1 > \eps > 0$, we have
$$
N(B_2^d, \|\cdot\|_2, \eps) \leq \left(\frac{3}{\eps}\right)^d.
$$
\end{lemma}

We are now prepared to prove Theorem~\ref{thm:gen}.

\begin{proof}[Proof of Theorem~\ref{thm:gen}]
For simplicity we take $m = d+1$.  Since $\R$ is a continuously differentiable manifold, so is $\R\backslash\{0\}$ and this implies that there are a countable number of closed\footnote{Note that in general these sets are open, but by writing each $\V_i$ as a countable union of closed sets (for example $\V_i = \cup_{j=1\ldots\infty} \phi^{-1}(\{x: \|x\|_2 \leq 1-1/j\})$) we observe that we can choose them to be closed. } sets $\V_i\subset \R\backslash\{0\}$ such that
\begin{itemize}
\item $\bigcup \V_i = \R\backslash\{0\}$
\item For each $\V_i$, there exists a $C^1$-diffeomorphism $\phi_i: \V_i \rightarrow B_2^d$.  In words, there is a homeomorphism $\phi_i$ from $\V_i$ to the unit Euclidean ball in $\mathbb{R}^d$ (denoted $B_2^d$) such that $\phi_i$ and $\phi_i^{-1}$ are continuously differentiable. 
\end{itemize}

Our strategy is to show that for fixed $i$, $0\notin \A(\V_i)$ with probability $1$.  We will then apply a union bound using the fact that there are only countably many $\V_i$.  

Fix an $i$, and for convenience set $\phi = \phi_i$ and $\V = \V_i$.  Since $\phi^{-1}$ is continuously differentiable, it is Lipschitz on the closed set $B^d_2$.  Thus there is an $L>0$ such that
\begin{equation}\label{eqn:Lip}
\|\phi^{-1}(x)-\phi^{-1}(y)\|_F \leq L\|x-y\|_2.
\end{equation}
Next, let $\overline{B_2^d}$ be an $(\eps/L)$-net for $B_2^d$ with cardinality at most $\left(\frac{3L}{\eps}\right)^d$.  This is of course possible by Lemma~\ref{lem:CN}.  Then the net $\overline{\V}$ defined by $\overline{\V} = \phi^{-1}(\overline{B_2^d})$ is an $\eps$-net for $\V$.  Indeed, for any $X\in \V$, we have $\phi(X) \in B_2^d$ and so there is a $\overline{b}\in \overline{B_2^d}$ such that
$$
\|\overline{b} - \phi(X)\|_2 \leq \frac{\eps}{L}.
$$
By~\eqref{eqn:Lip} we then have
$$
\|\phi^{-1}(\overline{b}) - X\|_F \leq L\cdot\|\overline{b}-\phi(X)\|_2 \leq L\cdot\frac{\eps}{L} = \eps.
$$
Since $\phi^{-1}(\overline{b}) \in \phi^{-1}(\overline{B_2^d})$, this shows that $\overline{\V}$ is an $\eps$-net for $\V$.  

Using the fact that $\overline{\V}$ is an $\eps$-net for $\V$, we have that for any $X\in \V$, there is an $\overline{X} \in \overline{\V}$ such that $\|X - \overline{X}\|_F \leq \eps$.  This then implies
\begin{align*}
\|\A(X)\|_\infty &\geq \|\A(\overline{X})\|_\infty - \|\A(X - \overline{X})\|_\infty \\ \notag
&\geq \|\A(\overline{X})\|_\infty - \|\A\|_{F\rightarrow\infty}\|X-\overline{X}\|_F \\ \notag
&\geq \|\A(\overline{X})\|_\infty - \eps\cdot\|\A\|_{F\rightarrow\infty},
\end{align*}

where $\|\cdot\|_{F\rightarrow\infty}$ denotes the operator norm from the Frobenius norm to the supremum norm, $\|\cdot\|_\infty$.  Optimizing over all $X\in \V$ and $\overline{X} \in \overline{\V}$ yields
$$
\inf_{X\in \V} \|\A(X)\|_\infty \geq \min_{\overline{X}\in \overline{\V}} \|\A(\overline{X})\|_\infty - \eps\cdot\|\A\|_{F\rightarrow\infty}.
$$

We can then bound the probability (over the random choice of $\A$) by:
\begin{align*}
\P\left(\inf_{X\in \V} \|\A(X)\|_\infty = 0\right) &\leq \P\left(\inf_{X\in \V} \|\A(X)\|_\infty \leq \eps\log(1/\eps)\right)\\ \notag
&\leq \P\left(\min_{\overline{X}\in \overline{\V}} \|\A(\overline{X})\|_\infty - \eps\cdot\|\A\|_{F\rightarrow\infty} \leq \eps\log(1/\eps)\right).\notag
\end{align*}

Conditioning on whether $\|\A\|_{F\rightarrow\infty} > \log(1/\eps)$ and using the law of total probability yields
\begin{align}\label{eqn:probs}
&\P\left(\min_{\overline{X}\in \overline{\V}} \|\A(\overline{X})\|_\infty - \eps\cdot\|\A\|_{F\rightarrow\infty} \leq \eps\log(1/\eps)\right)\notag \\ \leq 
 \; &\P\left(\min_{\overline{X}\in \overline{\V}} \|\A(\overline{X})\|_\infty \leq 2\eps\log(1/\eps)\right) + \P\Big(\|\A\|_{F\rightarrow\infty} > \log(1/\eps)\Big).
\end{align}

Clearly, for $\eps$ small, the second term in this last line of~\eqref{eqn:probs} is neglible.  Thus it remains to bound the first term.  Letting $z_1, \ldots, z_m$ be the coordinates of $\A(\overline{X})$ for a given $\overline{X} \in \overline{\V}$, we have: 
\begin{align*}
\P\left(\min_{\overline{X}\in \overline{\V}} \|\A(\overline{X})\|_\infty \leq 2\eps\log(1/\eps)\right) &\leq |\overline{\V}|\cdot \P\left(\|\A(\overline{X})\|_\infty \leq 2\eps\log(1/\eps)\right) \\
&= |\overline{\V}|\cdot \P\left(\max\{|z_1|, \ldots, |z_m|\} \leq 2\eps\log(1/\eps)\right)\\
&\leq \left(\frac{3L}{\eps}\right)^d\cdot \prod_{i=1}^m \Big(\P\left(|z_i| \leq 2\eps\log(1/\eps)\right)\Big),\\
\end{align*}

where in the last line we have utilized the independence of all $z_i = \langle A_i, \overline{X} \rangle$ and the size of the net $\V$.  
 
 Now, 
\begin{align*}
 \P\left(|z_i| \leq 2\eps\log(1/\eps)\right) &= \P\left(|\langle A_i, X\rangle| \leq 2\eps\log(1/\eps)\right) \\
 &=  \P\left(\left|\< A_i, \frac{X}{\|X\|_F}\>\right| \leq \frac{2\eps\log(1/\eps)}{\|X\|_F}\right).
\end{align*}
 
 Since $\V$ is closed and does not contain zero, the Frobenius norm of any $X\in\V$ is bounded uniformly away from zero.  This combined with the assumption that $\mathbb{P}(| \langle A_i, X \rangle | < \eps) < C\eps$ for every $X$ with $\|X\|_F=1$ yields

\begin{align*}
\left(\frac{3L}{\eps}\right)^d\cdot \prod_{i=1}^m \Big(\P\left(|z_i| \leq 2\eps\log(1/\eps)\right)\Big)
&\leq \left(\frac{3L}{\eps}\right)^d\cdot \left(4C'\eps\log(1/\eps)\right)^m\\
&= C''\eps^{m-d}\cdot\left(\log(1/\eps)\right)^m\\
&= C''\eps\cdot\left(\log(1/\eps)\right)^m,
\end{align*}

  where $C$, $C'$ and $C''$ are constants which do not depend on $\eps$.  The last equality follows since $m = d+1$.  Taking $\eps$ to zero once again makes this last term vanish.  Thus the probability that the null space of $\A$ intersects $\V$ is zero.  Since there are only countably many $\V_i$, the probability that the null space of $\A$ intersects any of these sets is also zero.  

\end{proof}

Now we turn to proving our main result Theorem~\ref{thm:main}.  
To prove this theorem, it will be useful to view the space of rank-$2r$ unit norm matrices as a smooth manifold.  Then Theorem~\ref{thm:main} will follow as a corollary of Theorem~\ref{thm:gen}.  We denote by $\|\cdot\|_F$ the usual Frobenius norm for matrices.

\begin{lemma}\label{lem:mani}
The space of rank-$r$ matrices with fixed Frobenius norm,
$$
\R = \{X\in\mathbb{R}^{n\times n} : \rank(X) = 2r, \|X\|_F=1\},
$$
is a smooth manifold with dimension $4nr-4r^2-1$.
\end{lemma}

\begin{proof}

We will first show that the space of rank-$2r$ matrices (with arbitrary Frobenius norm) is a smooth manifold.  This is a well-known result but we sketch the proof.  Then we will show that the intersection of this space and the sphere of all unit norm matrices is transverse which will yield the desired result.  To this end, let $\R'$ be as in~\eqref{defS} and let $\M = \{A\in \mathbb{R}^{n\times n}\}$ be the set of all $n\times n$ matrices.  Let $\G$ be the lie group consisting of the cross product of the general linear group with itself:
$$
\G = GL(n, \mathbb{R}) \times GL(n, \mathbb{R}).
$$ 

For an element $(g, h) \in \G$, let it act on elements $A\in \R'$ by $(g, h)A = gAh^{-1}$. Since this is just matrix multiplication, this action is clearly continuous.  Moreover, it is transitive.  Indeed, let $A, B\in \R'$  Since $A$ is rank $2r$, there are $g, h\in GL(n,\mathbb{R})$ such that 
$$
gAh^{-1} = \left[
     \begin{array}{lr}
       \mathbb{I}_{2r} & \\
        & \mathbb{O}_{n-2r}
     \end{array}
   \right] \defby D,
$$

where $\mathbb{I}_{2r} $ denotes the $2r\times 2r$ identity matrix and $\mathbb{O}_{n-2r}$ denotes the $n-2r \times n-2r$ matrix of zeros. Similarly, there are $y,z\in GL(n,\mathbb{R})$ such that $yBz^{-1}$ equals this same block matrix.  Thus $gAh^{-1} = yBz^{-1}$ and so $A=g^{-1}yBz^{-1}h$ which proves transitivity since $(g^{-1}y, z^{-1}h) \in \G$.  Next let $\H$ be the stabilizer of the matrix $D$ under the action of $\G$.  Since the action of $\G$ is continuous and transitive, the stabilizer $\H$ is a closed subgroup of the lie group $\G$ and thus $\H$ is a closed lie subgroup.  Therefore $\G/\H$ is a smooth manifold.  But $\H$ is the stabilizer of $D$ under $\G$ and so since the action is transitive, $\G/\H$ must be isomorphic to the orbit of $D$ under the action of $\G$.  But the orbit of $D$ is precisely our set $\R'$ and thus $\R'$ is also a smooth manifold. 

We now wish to show that $\R$ is also a smooth manifold by viewing $\R$ as the intersection of $\R'$ and the sphere
$$
\S = \{A\in \M : \|A\|_F = 1\}.
$$

It is clear that $\S$ is a smooth manifold (it is a sphere of the smooth manifold $\M$) and by above $\R'$ is also a smooth manifold.  Since $\R = \R'\cap \S$ and both $\R'$ and $\S$ are smooth manifolds, to show that $\R$ is also a smooth manifold it suffices to show that this intersection is transverse.  That is, we need to show that for any $A \in \R'\cap \S$, the direct sum of the tangent space of $\S$ at $A$ and $\R'$ at $A$ is equal to the tangent space of $\M$ of $A$:
$$
T_A(\S) \oplus T_A(\R') = T_A(\M).
$$

Since $\S$ has codimension $1$ in $\M$, it will suffice to show that $T_A(\R')$ contains a vector in the direction normal to the sphere.  Now for any $A\in \R'$ and $k\neq 0$, the matrix $kA$ will clearly also have rank $2r$ and thus be contained in $\R'$.  Therefore $\R'$ contains the entire line $L$ through the origin containing $A$ (excluding the actual origin itself).  Then since $L \subset \R'$, we have $L = T_A(L) \subset T_A(\R')$.  Thus we must indeed have $T_A(\S) \oplus T_A(\R') = T_A(\M)$.  Therefore the intersection $\R = \S\cap \R'$ is transverse and so it is a smooth manifold.

Finally, it is well known that the dimension of the manifold $\R'$ is $4nr-4r^2$ (see~\cite[Chapter 8]{lee2003introduction}) and the codimension of $\S$ in $\M$ is 1.  Thus codim($\R$) = codim($\S$) + codim($\R'$) = $n^2 - (4nr - 4r^2) + 1$ and so the dimension of $\R$ is $4nr - 4r^2 - 1$.  

\end{proof}

We finally show that Theorem~\ref{thm:main} and Theorem~\ref{thm:weak} follow as corollaries.

\begin{proof}[Proof of Theorem~\ref{thm:main}]
By Lemma~\ref{lem:mani}, $\R$ is a smooth manifold of dimension $d = 4nr - 4r^2 - 1$ and note that clearly $\R = \R\backslash\{0\}$.  Let $\A$ be the operator taking $m \geq 4nr - 4r^2$ Gaussian measurements $\langle A_i, X \rangle$ for $X\in \R$ and $A_i$ (for $i = 1, 2, \ldots m$) having i.i.d. Gaussian entries.  Then all $\langle A_i, X \rangle$ are independent and have (the same) continuous density.  Therefore by Theorem~\ref{thm:gen}, $\Null(\A) \cap \R = \emptyset$.  Applying Theorem~\ref{thm:gen} for all ranks between $1$ and $2r$, we see that there is no matrix of rank $2r$ or less in the null space of $\A$.  Thus when $M$ has rank $r$ (or less) there can be no other matrix $X$ with $\A(X) = \A(M)$ having the same or lower rank.  This proves that~\eqref{eqn:MCtheory} must recover the matrix $M$ and completes the proof.
\end{proof}

\begin{proof}[Proof of Theorem~\ref{thm:weak}]

Let $\W = \{X - M : \rank(X) = r\}$.  Note the proof of Lemma~\ref{lem:mani} explicitly shows that the space of all matrices of a fixed rank $r$ is a smooth manifold of dimension $2nr - r^2$.  Since $\W$ is a shift of this space, it is also a smooth manifold of the same dimension.  Then by Theorem~\ref{thm:gen}, we have that with probability one
$$
\W\backslash\{0\} \cap \Null(\A) = \emptyset.
$$
Repeating this for ranks $1$ through $r$, we get that with probability one
\begin{equation}\label{yes}
\W' \backslash\{0\} \cap \Null(\A) = \emptyset
\end{equation}
where $\W' = \{X - M : \rank(X) \leq r\}$.  Now let $X$ be the solution of the rank minimization problem~\eqref{eqn:MCtheory}.  Since $M$ has rank $r$ and is a feasible matrix, $\rank(X) \leq r$ as well.  Thus $X - M \in \W'$.  But since $\A(X) = \A(M)$, $X-M\in\Null(A)$.  Thus by~\eqref{yes} it must be that $X - M = 0$ which shows $X=M$ is the recovered matrix.
\end{proof}

\section{Discussion}\label{sec:discuss}
The bounds on the number of measurements given by Theorems~\ref{thm:main} and~\ref{thm:weak} of $4nr-4r^2$ and $2nr-r^2+1$ are analagous to the bounds of $2s$ and $s+1$ in compressed sensing.  As we did in the compressed sensing case, it is of course insightful to compare rank minimization and nuclear-norm minimization.  To (provably) recover $n\times n$ rank-$r$ matrices using nuclear-norm minimization, one needs $Cnr$ measurements.  As discussed in~\cite{candes2009tight}, by observing that the space of rank-$r$ matrices has a subspace which consists of all rank-$r$ matrices whose last $n-r$ rows are zero, one sees that at least $2nr$ measurements are required to recover all rank-$r$ matrices.  In~\cite{oymak2010new} explicit formulas and graphs are given from which bounds on the constant $C$ can be derived.  Even more recent results in~\cite{OH10:newest} prove that $6nr$ measurements suffice for weak recovery and $16nr$ measurements suffice for strong recovery.  New work in~\cite{parrilo2010convex,bentalk} also shows weak recovery when $m\geq 6nr-3r^2$.  
In addition, numerical results indicate that weak recovery requires about $4nr - 2r^2$ Gaussian measurements~\cite[Figure 1]{oymak2010new}.
Thus according to these results, rank minimization does succeed with somewhat fewer measurements.
We emphasize that this should not be a surprise --- nuclear-norm minimization is a tractible method whereas rank minimization is an intractible method whose guarantees give us theoretical bounds with which to compare.  In fact, the price to pay for a tractible method in low-rank matrix recovery seems to be a very reasonable one.  

As discussed above, our general manifold result, Theorem~\ref{thm:gen}, is tight.  However, this does not imply that its consequences, Theorems~\ref{thm:main} and~\ref{thm:weak}, are tight since the set of matrices of fixed rank is not a linear subspace.  We conjecture that the strong recovery requirement, $m\geq 4nr-4r^2$ from Theorem~\ref{thm:main}, is tight because the number of measurements required matches the dimension of the underlying manifold.  In the case of the weak recovery requirement $m\geq 2nr -2r^2 + 1$ given by Theorem~\ref{thm:weak}, we require $m$ to be one greater than the dimension of the underlying manifold.  However, we once again conjecture this to be tight at least within an additive factor of one for the same reason.


Our results in conjunction with work on nuclear-norm minimization show how close nuclear-norm minimization guarantees are to those of the intractible problem of rank minimization.  While rank minimization requires fewer measurements, it is not at all an unreasonable amount to pay in order to solve the problem via a computationally feasible method.

 \subsection*{Acknowledgment}
We would like to thank Boris Bertman and Rohit Thomas for thoughtful discussions.  This work was partially supported by the NSF DMS
EMSW21-VIGRE grant.

\bibliography{thebib}

\begin{thebibliography}{10}

\bibitem{B07:Compressive}
R.~Baraniuk.
\newblock Compressive sensing.
\newblock {\em IEEE Sig. Proc. Mag.}, 24(4):118--121, 2007.

\bibitem{cai2008singular}
J.F. Cai, E.~J. Candes, and Z.~Shen.
\newblock {A singular value thresholding algorithm for matrix completion}.
\newblock {\em SIAM J. Optim.}, 20(4):1956--1982, 2008.

\bibitem{candes2009tight}
E.~J. Candes and Y.~Plan.
\newblock {Tight oracle bounds for low-rank matrix recovery from a minimal
  number of random measurements}.
\newblock {\em IEEE Trans. Inform. Theory}, 2009.
\newblock to appear.

\bibitem{candes2010matrix}
E.~J. Candes and Y.~Plan.
\newblock {Matrix completion with noise}.
\newblock {\em Proc. of the IEEE}, 98(6):925--936, 2010.

\bibitem{candes2009exact}
E.~J. Cand{\`e}s and B.~Recht.
\newblock {Exact matrix completion via convex optimization}.
\newblock {\em Foundations of Computational Mathematics}, 9(6):717--772, 2009.

\bibitem{CRT06:Robust-Uncertainty}
E.~J. Cand{\`e}s, J.~Romberg, and T.~Tao.
\newblock Robust uncertainty principles: {E}xact signal reconstruction from
  highly incomplete {F}ourier information.
\newblock {\em IEEE Trans. Info. Theory}, 52(2):489--509, Feb. 2006.

\bibitem{CT05:Decoding}
E.~J. Cand\`es and T.~Tao.
\newblock Decoding by linear programming.
\newblock {\em IEEE Trans. Inform. Theory}, 51:4203--4215, 2005.

\bibitem{CW08:Intro}
E.~J. Cand\`es and M.~Wakin.
\newblock An introduction to compressive sampling.
\newblock {\em IEEE Signal Proc. Mag.}, 25(2):21--30, 2008.

\bibitem{CSwebpage}
Compressed sensing webpage.
\newblock \verb=http://www.dsp.ece.rice.edu/cs/=.

\bibitem{dai2010set}
W.~Dai and O.~Milenkovic.
\newblock {SET: an algorithm for consistent matrix completion}.
\newblock In {\em ICASSP 2010 IEEE Int. Conf.}, pages 3646--3649. IEEE, 2010.

\bibitem{DS89:Uncertainty-Principles}
D.~L. Donoho and P.~B. Stark.
\newblock Uncertainty principles and signal recovery.
\newblock {\em SIAM J. Appl. Math.}, 49(3):906--931, June 1989.

\bibitem{DT06:Counting-Faces}
D.~L. Donoho and J.~Tanner.
\newblock {\em J. of the AMS}, 22.

\bibitem{G09:Recovering}
D.~Gross.
\newblock Recovering low-rank matrices from few coefficients in any basis.
\newblock to appear, 2009.

\bibitem{lee2003introduction}
J.M. Lee.
\newblock {\em {Introduction to smooth manifolds}}.
\newblock Springer Verlag, 2003.

\bibitem{lee2010admira}
K.~Lee and Y.~Bresler.
\newblock {Admira: Atomic decomposition for minimum rank approximation}.
\newblock {\em Information Theory, IEEE Transactions on}, 56(9):4402--4416,
  2010.

\bibitem{LGM96:Const}
G.~G. Lorentz, M.~von Golitschek, and Y.~Makovoz.
\newblock {\em Constructive Approximation: Advanced Problems}, volume 304.
\newblock Springer, Berlin, 1996.

\bibitem{ma2009fixed}
S.~Ma, D.~Goldfarb, and L.~Chen.
\newblock {Fixed point and Bregman iterative methods for matrix rank
  minimization}.
\newblock {\em Mathematical Programming}, pages 1--33, 2009.

\bibitem{meka2009guaranteed}
R.~Meka, P.~Jain, and I.S. Dhillon.
\newblock {Guaranteed rank minimization via singular value projection}.
\newblock Preprint, 2009.

\bibitem{M05:Data}
S.~Muthukrishnan.
\newblock {\em Data Streams: Algorithms and Applications}.
\newblock Now Publishers, Hanover, MA, 2005.

\bibitem{negahban2010unified}
S.~Negahban, P.~Ravikumar, M.J. Wainwright, and B.~Yu.
\newblock {A unified framework for high-dimensional analysis of $ M
  $-estimators with decomposable regularizers}.
\newblock 2010.
\newblock Submitted.

\bibitem{oymak2010new}
S.~Oymak and B.~Hassibi.
\newblock {New Null Space Results and Recovery Thresholds for Matrix Rank
  Minimization}.
\newblock Submitted, 2010.

\bibitem{OH10:newest}
S.~Oymak and B.~Hassibi.
\newblock {Tight Recovery Thresholds and Robustness Analysis for Matrix Rank
  Minimization}.
\newblock In preparation, 2010.

\bibitem{parrilo2010convex}
P.A. Parrilo, A.S. Willsky, V.~Chandrasekaran, and B.~Recht.
\newblock {The Convex Geometry of Linear Inverse Problems}.
\newblock Submitted, 2010.

\bibitem{bentalk}
B.~Recht, Mar. 2011.
\newblock Personal communication.

\bibitem{recht2007guaranteed}
B.~Recht, M.~Fazel, and P.A. Parrilo.
\newblock {Guaranteed minimum-rank solutions of linear matrix equations via
  nuclear norm minimization}.
\newblock Preprint, 2007.

\end{thebibliography}

\end{document}